\documentclass[a4paper]{article}

\usepackage{lipsum}
\usepackage{amsmath,amsthm,amssymb}
\usepackage{graphicx}
\usepackage{hhline,url}
\usepackage{authblk}
\usepackage{harvard}
\usepackage{color}

\citationmode{abbr}

\newtheorem{thm}{Theorem}[section]
\newtheorem{lem}[thm]{Lemma}
\newtheorem{remark}[thm]{Remark}
\newenvironment{rem}{\begin{remark}\rm}{\end{remark}}
\newtheorem{prop}[thm]{Proposition}
\newtheorem{Definition}[thm]{Definition}
\newenvironment{dfn}{\begin{Definition}\rm}{\end{Definition}}
\newtheorem{Corollary}[thm]{Corollary}

\newtheorem{Example}{Example}[section]
\newenvironment{eg}{\begin{Example}\rm}{\end{Example}}

\renewcommand{\H}{\mathbb{H}}
\newcommand{\D}{\mathcal{D}}
\newcommand{\R}{\mathbb{R}}
\newcommand{\dd}{\,\mathrm{d}}
\newcommand{\ve}{\varepsilon}
\newcommand{\M}{\mathcal{M}}
\newcommand{\kk}[2]{\langle #1, #2 \rangle_K}
\newcommand{\F}{\mathcal{F}}
\renewcommand{\S}{\mathcal{S}}

\makeatletter

    \@addtoreset{equation}{section}
\makeatother

\title{Convergence analysis of approximation formulas for analytic functions via duality for potential energy minimization}

\author[1]{Satoshi Hayakawa\footnote{hayakawa@maths.ox.ac.uk}}
\author[2]{Ken'ichiro Tanaka\footnote{kenichiro@mist.i.u-tokyo.ac.jp}}

\affil[1]{Mathematical Institute, University of Oxford}

\affil[2]{
Department of Mathematical Informatics,
Graduate School of Information Science and Technology,
The University of Tokyo}

\date{}

\providecommand{\keywords}[1]{\textbf{\textit{Keywords---}} #1}

\begin{document}

\maketitle

\begin{abstract}
	We investigate the approximation formulas
	that were proposed by 
	\citeasnoun{tanaka2018design},
	in weighted Hardy spaces, which are analytic function spaces with certain asymptotic decay.
	Under the criterion of minimum worst error of $n$-point approximation formulas,
	we demonstrate that the formulas are nearly optimal.
	We also obtain the upper bounds of the approximation errors
	that coincide with the existing heuristic bounds in asymptotic order
	by duality theorem for the minimization problem of potential energy.
	
	\noindent\keywords{approximation, weighted Hardy space, convex optimization, duality, discrete energy minimization}
\end{abstract}

\section{Introduction}\label{sec:introduction}


By taking over the arguments of \citeasnoun{tanaka2017potential}, 
\citeasnoun{tanaka2018design} 
proposed an algorithm to design accurate approximation formulas
in function spaces called weighted Hardy spaces defined by 
\begin{align}
	\H^\infty(\D_d, w)
	:=\left\{
		f:\D_d\to\mathbb{C}
		\,\middle|\,
		\text{$f$ is analytic on $\D_d$},\ \sup_{z\in\D_d}\left|\frac{f(z)}{w(z)}\right|<\infty
	\right\},
	\label{eq:Hinf_first}
\end{align}
where $d > 0$, $\D_d:=\{z\in\mathbb{C}\mid |\mathop{\mathrm{Im}}z|<d\}$, and
$w$ is a weight function characterized later in Section~\ref{2-1}.  
The spaces $\H^\infty(\D_d, w)$ are often considered as spaces of transformed functions 
for well-used sinc approximation formulas shown later in \eqref{eq:sinc_approx}. 
The objective of \citeasnoun{tanaka2017potential} and \citeasnoun{tanaka2018design} 
was to provide formulas ourperforming the sinc formulas. 
However, 
their studies only provided heuristic analyses on the proposed formulas without any theoretical guarantees, 
although their methods have shown superiority to the sinc approximation formulas. 
In this study,
we mathematically
\begin{itemize}
	\item[(1)]
		prove near optimality of the formulas, and
	\item[(2)]
		provide a general upper bound of the errors of the proposed formulas 
		and show that the bound coincides in asymptotic order with the heuristic bound derived
		by \citeasnoun{tanaka2017potential}.
\end{itemize}


Below we describe the background of this study more precisely.  
The spaces $\H^\infty(\D_d, w)$ appear in literature
as spaces of variable-transformed functions
\cite{stenger1993,stenger2011,sugihara2003,tanaka2009function}.
For example,
the double exponential (DE) transform, which is well-used in numerical analysis
\cite{takahashi1974double},
has the form
\[
	f(x)=g\left(\tanh\left(\frac\pi2\sinh(x)\right)\right)
\]
and shows a double-exponential decay.
Also, TANH transform $g(\tanh(x/2))$ is commonly used
\cite{schwartz1969numerical,haber1977tanh}.
These variable transformations are employed for the accurate approximation of functions
by yielding functions with rapid decay on $\D_d$,
which enables us to neglect the values of the functions for large $|x|$.
This motivates us to analyze the approximation possibility over weighted Hardy spaces
with general weight functions $w$.
After \citeasnoun{sugihara2003} demonstrated near optimality
of sinc approximation formulas
\begin{align}
f(x) \approx \sum_{k=N_-}^{k=N_+}f(kh)\mathop\mathrm{sinc}\left(\frac{x}h-k\right)
\label{eq:sinc_approx}
\end{align}
for several weight functions $w$,
attempts to construct an optimal formula
for general weight functions
was started in the literature.

For this purpose, \citeasnoun{tanaka2017potential} employed potential theoretical arguments
to generate sampling points for the approximation of functions.
Furthermore, \citeasnoun{tanaka2018design} simplified the arguments
and proposed accurate formulas $L_n[a^{\ast};f](x)$ given later by~\eqref{eq:0-2-1} with special sets $a^{\ast}$ of sampling points. 
The formulas $L_n[a^{\ast};f](x)$ outperform the sinc methods 
for functions $f \in \H^\infty(\D_d, w)$. 
The authors showed that 
\begin{align}
	\sup_{\|f\|\le1,\ x\in\R}|f(x)-L_n[a^*;f](x)|
	\le \exp\left(-\frac{F^\mathrm{D}_{K, Q}(n)}{n-1}\right),
	\label{eq:RH_first}
\end{align}
where 
$\| f \|$ is a norm of $f \in \H^\infty(\D_d, w)$ and  
$F^\mathrm{D}_{K, Q}(n)$ is determined later in~\eqref{eq:def_disc_pot} by a ``discrete'' energy minimization problem.
Furthermore, 
they considered 
the minimum worst error $E_n^{\min}(\H^\infty(\D_d, w))$ 
in \eqref{eq:def_min_wor_err}
of $n$-point approximation formulas in $\H^\infty(\D_d, w)$
and evaluated it as 
\begin{align}
	\exp\left(-\frac{F^\mathrm{C}_{K, Q}(n)}n\right)
	\le E_n^{\min}(\H^\infty(\D_d, w)), 
\label{eq:min_wor_err_lb_first}
\end{align}
where $F^\mathrm{C}_{K, Q}(n)$ is determined later in~\eqref{eq:def_cont_pot} by a ``continuous counterpart'' of the above energy minimization problem. 
The following problems about the formula $L_n[a^*;f](x)$
were unsolved in \citeasnoun{tanaka2018design}. 
\begin{enumerate}
\item[(i)]
Since 
(the RHS of \eqref{eq:min_wor_err_lb_first})$\, \le \,$(the LHS of \eqref{eq:RH_first}), 
the formula $L_n[a^*;f](x)$ is assured of ``near optimality''
if $F^\mathrm{C}_{K, Q}(n)$ and $F^\mathrm{D}_{K, Q}(n)$ are close.
However, 
their difference was not estimated. 

\item[(ii)]
To estimate the convergence rate of the error in the LHS of \eqref{eq:RH_first}, 
we need to know how $F^\mathrm{D}_{K, Q}(n)$ depends on $n$. 
However, it was not known. 

\end{enumerate}

In this paper, 
we provide solutions to these problems.
Our contributions (1) and (2) mentioned in the first paragraph of this section 
correspond to the solutions to problems (i) and (ii),  respectively. 
More precisely, 
we show the following statements. 
\begin{enumerate}
\item[(1)]
We show an evaluation like
\[
	F_{K, Q}^\mathrm{D}(n)\lesssim F_{K, Q}^\mathrm{C}(n)\lesssim 2 F_{K, Q}^\mathrm{D}(n).
\]
Its rigorous version is given by Theorem~\ref{main1} in Section~\ref{sec:main_results}. 
The quantities $F^\mathrm{D}_{K, Q}(n)$ and $F^\mathrm{C}_{K, Q}(n)$
were obtained from the optimal solutions of the ``discrete'' energy minimization problem and its ``continuous counterpart'',
respectively. 
Therefore we construct a feasible solution for the latter using the optimal solution of the former to show this theorem. 

\item[(2)]
We show an inequality
	\[
		\frac{F_{K, Q}^\mathrm{C}(n)}{n}\ge\frac{Q(\alpha_n)}2, 
	\]
where 
$Q(x) = - \log w(x)$ and $\alpha_{n}$ is determined by a tractable inequality. 
Its details are given by Theorem~\ref{main2} in Section~\ref{sec:main_results}. 
By combining this inequality, the above statement (1), and Inequality~\eqref{eq:RH_first}, 
we obtain explicit convergence rates of the proposed formulas. 
To show this theorem, 
we consider the dual problem of the ``continuous'' energy minimization problem and 
provide its feasible solution. 
For preparation, we present a primal-dual theory of the energy minimization problem
in Section~\ref{sec:duality_theorem}. 

\end{enumerate}
As a result, 
we explicitly obtain lower bounds of $F_{K, Q}^\mathrm{C}(n)$
and demonstrate that
the rates of lower bounds coincide with those of heuristic bounds in \citeasnoun{tanaka2017potential}.

The rest of this paper is organized as follows. 
In Section \ref{sec2},
we present a mathematical overview of the existing studies
and describe our main results as mathematical statements.
Section \ref{sec3} describes the proof of the first result,
i.e., Theorem~\ref{main1}.
Section \ref{sec:duality_theorem} contains general arguments,
which introduce the concept of ``positive semi-definite in measure''.
Then, we show that the problem under our interest is a special case of that concept
and derive the duality theorem.
The evaluations for the second result,
described by Theorem~\ref{main2},
are given in Section \ref{sec5}.
We compare the bounds with those in \citeasnoun{tanaka2017potential} in Section \ref{sec:numerical_examples}. 
Finally, we describe the concluding remarks in Section \ref{sec:conclusion}.


\section{Mathematical preliminaries and main results}\label{sec2}

\subsection{General settings}\label{2-1}

We first give some definitions and formulate the problem mathematically.
Let $d>0$ and define the strip region $\D_d:=\{z\in\mathbb{C}\mid |\mathop{\mathrm{Im}}z|<d\}$.
Throughout this paper,
a {\it weight function} $w:\D_d\to\mathbb{C}$ is supposed to satisfy the following conditions:
\begin{enumerate}
	\item
		$w$ is analytic and does not vanish over the domain $\D_d$
		and takes values in $(0, 1]$ on $\R$;
	\item
		$w$ satisfies $\lim_{x\to\pm\infty}\int_{-d}^d|w(x+iy)|\dd y=0$
		and $\lim_{y\nearrow d}\int_{-\infty}^\infty (|w(x+iy)|+|w(x-iy)|)\dd x<\infty$;
	\item
		$\log w$ is strictly concave on $\R$.
\end{enumerate}
For a weight function with the above conditions,
we define the weighted Hardy space $\H^\infty(\D_d, w)$ on $\D_d$ in \eqref{eq:Hinf_first}.
Recall that it is defined by
\begin{align}
	\H^\infty(\D_d, w)
	:=\left\{
		f:\D_d\to\mathbb{C}
		\,\middle|\,
		\text{$f$ is analytic on $\D_d$},\ \sup_{z\in\D_d}\left|\frac{f(z)}{w(z)}\right|<\infty
	\right\}.
	\label{eq:Hinf}
\end{align}
We define
\[
	\|f\|:=\sup_{z\in\D_d}\left|\frac{f(z)}{w(z)}\right|
\]
for $f\in\H^\infty(\D_d,w)$,
and the expression $\|f\|<\infty$ shall also imply $f\in\H^\infty$ in the following.

For an approximation formula over $\H^\infty(\D_d, w)$,
an evaluation criterion needs to be defined.
Based on \citeasnoun{sugihara2003} and \citeasnoun{tanaka2018design},
we adopt the minimum worst-case error
\begin{align}
	&E_n^{\min} (\H^\infty(\D_d, w)) \notag \\
	&:=\inf\left\{
		\sup_{\|f\|\le 1,\ x\in\R}\left|
			f(x)-\sum_{j=1}^l\sum_{k=0}^{n_j-1}f^{(k)}(a_j)\phi_{jk}(x)
		\right|
		\,\middle|\,
		\begin{array}{c}
			1\le l\le n,\ m_1+\cdots+m_l=n,\\
			a_j\in\D_d\ \text{are distinct},\\
			\phi_{jk}:\D_d\to\mathbb{C}\ \text{are analytic}
		\end{array}
	\right\}
	\label{eq:def_min_wor_err}
\end{align}
as the optimal performance over all possible $n$-point interpolation formulas on $\R$,
which is applicable to any $f\in\H^\infty(\D_d, w)$.

\subsection{Properties of approximation formulas to be analyzed}

Let us introduce some functions dependent on
an $n$-sequence $a=\{a_j\}_{j=1}^n\subset\R$ as follows.
\begin{align*}
	T_d(x)&:=\tanh\left(\frac\pi{4d}x\right),\\
	B_n(x;a,\D_d)&:=\prod_{j=1}^n\frac{T_d(x)-T_d(a_j)}{1-T_d(a_j)T_d(x)},\\
	B_{n;k}(x, a, \D_d)&:=\prod_{\substack{1\le j\le n,\\
	j\ne k}}\frac{T_d(x)-T_d(a_j)}{1-T_d(a_j)T_d(x)}.
\end{align*}
Using these functions,
we can give an $n$-point interpolation formula
\begin{align}
	L_n[a;f](x):=\sum_{k=1}^nf(a_k)\frac{B_{n;k}(x; a,\D_d)w(x)}{B_{n;k}(a_k; a,\D_d)w(a_k)}
	\frac{T_d'(x-a_k)}{T_d'(0)},
	\label{eq:0-2-1}
\end{align}
which is known to characterize the value $E_n^{\min}(\H^\infty(\D_d, w))$
as follows.
\begin{prop}{\rm\cite{sugihara2003,tanaka2018design}}
	We have an upper bound of the error of (\ref{eq:0-2-1}) as
	\[
		\sup_{\|f\|\le 1,\ x\in\R}|f(x)-L_n[a;f](x)|
		\le \sup_{x\in\R}|B_n(x;a,\D_d)w(x)|
	\]
	for any fixed  sequence $a=\{a_j\}_{j=1}^n\subset\R$ (of distinct points).
	Moreover, by taking infimum of the above expression over all $n$-sequences,
	it holds that
	\[
		E_n^{\min}(\H^\infty(\D_d, w))
		=\inf_{a_j\in\R}
		\sup_{\|f\|\le1,\ x\in\R}|f(x)-L_n[a;f](x)|
		=\inf_{a_j\in\R}\sup_{x\in\R}|B_n(x;a,\D_d)w(x)|.
	\]
\end{prop}

By this assertion,
it is enough to consider interpolation formulas of the form (\ref{eq:0-2-1}).
Additionally,
this motivates us to analyze the value $\sup_{x\in\R}|B_n(x;a,\D_d)w(x)|$,
which is simpler than the worst-case error of (\ref{eq:0-2-1}).
In \citeasnoun{tanaka2017potential} and \citeasnoun{tanaka2018design},
\[
	-\log\left(\inf_{a_j\in\R}\sup_{x\in\R}|B_n(x;a,\D_d)w(x)|
	\right)
\]
is treated as an optimal value of an optimization problem
(justifiable by the addition rule of $\tanh$)
\[
	(\mathrm{DC})\quad
	\begin{array}{lr}\text{maximize}&\displaystyle \inf_{x\in\R}
	\left(\sum_{i=1}^nK(x-a_i)+Q(x)\right)\\
	\text{subject to}&\displaystyle a_1<\cdots<a_n,
	\end{array}
\]
where $K$ and $Q$ are defined by
\begin{align}
	K(x) &:=-\log |T_d(x)|\ \left(=-\log\left|\tanh\left(\frac\pi{4d}x\right)\right|\right), \\
	Q(x) &:=-\log w(x).
\end{align}
They considered a continuous relaxation of (DC)
as 
\[
	(\mathrm{CT})\quad
	\begin{array}{lr}\text{maximize}&\displaystyle \inf_{x\in\R}
	\left(\int_\R K(x-y)\dd\mu(y)+Q(x)\right)\\
	\text{subject to}&\displaystyle \mu\in\M_c(\R, n),
	\end{array}
\]
where, we define $\M(\R, n)$ as the set of all (positive) Borel measures $\mu$ over $\R$
with $\mu(\R)=n$ and
\[
	\M_c(\R, n):=\{\mu\in\M(\R, n)\mid \mathop{\mathrm{supp}}\mu\ \text{is compact}\}.
\]
Because each feasible solution of (DC)
can be interpreted as a combination of $\delta$-measures
being a feasible solution of (CT),
\begin{align}
	\text{(the optimal value of (DC))} \le
	\text{(the optimal value of (CT))}
	\label{eq:LH}
\end{align}

Potential theoretical arguments \cite{saff1997,levin2001,tanaka2018design}
lead to the following proposition.
\begin{prop}{\rm\cite[Theorem 2.4, 2.5]{tanaka2018design}}\label{prop2-2}
	The energy of $\mu\in\M(\R, n)$ is defined as
	\begin{align}
		I_n^\mathrm{C}(\mu)
		:=\int_\R\int_\R K(x-y)\dd\mu(x)\dd\mu(y)
		+2\int_\R Q(x)\dd\mu(x).
	\end{align}
	Then, there exists a unique minimizer $\mu_n^*$ over $\M(\R, n)$ of $I_n^\mathrm{C}(\mu)$
	with a compact support
	and $\mu_n^*$ is also an optimal solution of {\rm(CT)}.
	Furthermore,
	if we define
	\begin{align}
		F_{K,Q}^\mathrm{C}(n)
		& :=I_n^\mathrm{C}(\mu_n^*)-\int_\R Q(x)\dd\mu_n^*(x) 
		\notag \\
		& \ \left(=\int_\R\int_\R K(x-y)\dd\mu_n^*(x)\dd\mu_n^*(y)
		+\int_\R Q(x)\dd\mu_n^*(x)\right),
		\label{eq:def_cont_pot}
	\end{align}
	the optimal value of (CT)
	coincides with
	$\displaystyle\frac{F_{K,Q}^\mathrm{C}(n)}n$.
\end{prop}

Following this proposition,
\citeasnoun{tanaka2018design} considered a discrete counterpart of
$I^\mathrm{C}_n(\mu)$ and $F_{K,Q}^\mathrm{C}$,
which are defined for $a=\{a_i\}_{i=1}^n$ ($a_1<\cdots<a_n$) as
\begin{align}
	& I_{K,Q}^\mathrm{D}(a)
	:=\sum_{i\ne j}K(a_i-a_j)+\frac{2(n-1)}n\sum_{i=1}^nQ(a_i), \label{eq:disc_ene} \\
	& F_{K,Q}^\mathrm{D}(n)
	:=I_{K, Q}^\mathrm{D}(a^*)-\frac{n-1}n\sum_{i=1}^nQ(a_i^*), \label{eq:def_disc_pot}
\end{align}
where $a^*=\{a_i^*\}_{i=1}^n$ is the unique minimizer of $I_{K, Q}^\mathrm{D}(a)$,
which certainly exists according to Theorem 3.3 in \citeasnoun{tanaka2018design}.
We can easily obtain $a^*$ numerically as it is a solution of the convex programming
and it is known to satisfy \cite[Theorem 4.1]{tanaka2018design}
\begin{align}
	\sup_{\|f\|\le1,\ x\in\R}|f(x)-L_n[a^*;f](x)|
	\le \exp\left(-\frac{F^\mathrm{D}_{K, Q}(n)}{n-1}\right).
	\label{eq:RH}
\end{align}
Then $E_n^{\min}(\H^\infty(\D_d, w))$ is evaluated as \cite[Remark 4.2]{tanaka2018design}
\[
	\exp\left(-\frac{F^\mathrm{C}_{K, Q}(n)}n\right)
	\le E_n^{\min}(\H^\infty(\D_d, w))
	\le \exp\left(-\frac{F^\mathrm{D}_{K, Q}(n)}{n-1}\right).
\]
Indeed, the left inequality holds true by (\ref{eq:LH}) and Proposition \ref{prop2-2}
and the right inequality follows from (\ref{eq:RH}).
By this evaluation,
we can consider $L_n[a^*;f](x)$ as a nearly optimal approximation formula
if $F^\mathrm{C}_{K, Q}(n)/n$ and $F^\mathrm{D}_{K, Q}(n)/(n-1)$ are sufficiently close.

\subsection{Main results}
\label{sec:main_results}

In this paper,
we demonstrate the following two theorems.
The first and second theorems, respectively, correspond to (1) and (2) in Section \ref{sec:introduction}.
\begin{thm}\label{main1}
	For $n\ge2$,
	the following holds true
	\[
		\frac{F^\mathrm{D}_{K,Q}(n)}{n-1}\le
		\frac{F_{K, Q}^\mathrm{C}(n)}n
		\le \frac{n}{n-1}\left(\frac{2F^\mathrm{D}_{K, Q}(n)}{n-1}
		+(3+\log2)\right).
	\]
\end{thm}

\begin{thm}\label{main2}
	Suppose $w$ is even on $\R$.
	For $\alpha_n>0$ that satisfies
	\[
		\frac{2\alpha_n}{\pi\tanh(d)}\frac{Q(\alpha_n)^2+Q'(\alpha_n)^2}{Q(\alpha_n)}\le n,
	\]
	we have
	\[
		\frac{F_{K, Q}^\mathrm{C}(n)}{n}\ge\frac{Q(\alpha_n)}2.
	\]
\end{thm}

Theorem \ref{main1} shows the near optimality of the approximation formula
$L_n[a^*;f](x)$.
By the assertion of the theorem,
we have, for arbitrary $\ve>0$,
\[
	\sup_{\|f\|\le1,\ x\in\R}|f(x)-L_n[a^*;f](x)|
	\le \sqrt{2e^3}E_n^{\min}(\H^\infty(\D_d, w))^{\frac1{2+\ve}}
\]
for each sufficiently large $n$.
In addition, Theorem \ref{main2} (combined with Theorem \ref{main1}) gives an explicit upper bound of
$E_n^{\min}(\H^\infty(\D_d, w))$ as
\[
	E_n^{\min}(\H^\infty(\D_d, w))\le
	\sqrt{2e^3}\exp\left(-\frac{n-1}{4n}Q(\alpha_n)\right).
\]

\subsection{Basic ideas to show the main results}

The left inequality of Theorem \ref{main1} is from Theorems 3.4 and 3.5 in \citeasnoun{tanaka2018design}.
To prove the right inequality of Theorem \ref{main1}, 
we consider the optimization problem
\begin{align}
	(\mathrm{P})\quad
	\begin{array}{lr}\text{minimize}&\displaystyle
	\int_\R\int_\R K(x-y)\dd\mu(x)\dd\mu(y) + 2\int_\R Q(x)\dd\mu(x)
	\nonumber\\
	\text{subject to}&
	\begin{array}{r}
	\mu \in \mathcal{M}(\mathbb{R}, n),
	\end{array}
	\end{array}
\end{align}
whose solution provides $F^\mathrm{C}_{K, Q}(n)$ as shown in Proposition \ref{prop2-2}. 
The quantity 
$F^\mathrm{D}_{K, Q}(n)$
is obtained from the optimal solution of a discrete counterpart of (P) given by~\eqref{eq:disc_ene}. 
Then, 
we construct a feasible solution of (P) given later by \eqref{eq:feasible_solution_of_P}
from the optimal solution of the discrete counterpart.
By using the feasible solution, 
we bound $F^\mathrm{C}_{K, Q}(n)$ from above by using $F^\mathrm{D}_{K, Q}(n)$. 

To prove Theorem \ref{main2}, 
we need a lower bound of the optimal value of (P).
However, because (P) is a minimization problem,
any concrete feasible solution does not help us.
Therefore,
we prove that
(P) can be regarded as an infinite-dimensional convex quadratic programming,
as $K$ is {\it positive semi-definite in measure} (Definition \ref{def2-1}),
and take the dual problem \cite{dorn1960duality,luenberger1997}.
We also show that the dual problem
\begin{align}
	(\mathrm{D})\quad
	\begin{array}{lr}\text{maximize}&\displaystyle -\int_\R\int_\R K(x-y)\dd\nu(x)\dd\nu(y)
	+ 2ns \\
	\text{subject to}&
	\begin{array}{r}
		\text{$\nu$ is a signed Borel measure}\\
		\displaystyle s-\int_\R K(\cdot-y)\dd\nu(y)\le Q
	\end{array}
	\end{array}
	\label{eq:pre_problem_D}
\end{align}
satisfies the weak and strong duality (Theorem \ref{thm2-4}),
i.e., the optimal value of (D) coincides with that of (P).
By this,
we can obtain a lower bound for the optimal value of (P),
taking concrete $\nu$ and $s$.
The practical advantage of taking (D)
is that $\nu$ can be a signed measure (though we indeed deal with a little wider class in Section \ref{sec:duality_theorem}),
which means that we can define $\nu$ as some Fourier transform of the symmetric function,
without confirming the non-negativity.
This solves one of the improper points of the evaluation in \citeasnoun{tanaka2017potential}.

\begin{remark}
Problem (D) in \eqref{eq:pre_problem_D} needs to be more rigorous to realize a primal-dual theory for (P) and (D). 
In Section \ref{sec:duality_theorem}, 
we provide a rigorous form of (D) by introducing a set $\mathcal{S}_{K}$ for $\nu$.
\end{remark}


\section{Proof of Theorem \ref{main1}}\label{sec3}

To prove Theorem \ref{main1}, we prepare the following lemmas.

\begin{lem}\label{lem:1-1}
	For arbitrary $t>0$,
	the following holds true.
	\[
		\int_0^1K(tx)\dd x \le K(t)+1.
	\]
\end{lem}

\begin{proof}
	Consider the function
	$g(x):=K(x)+\log\left(\frac{\pi}{4d}x\right)$ defined for $x>0$.
	We first prove that $g(x)$ is strictly increasing and satisfies $\lim_{x\searrow0}g(x)=0$.
	Let $h(x):=\exp\left(g\left(\frac{2d}\pi x\right)\right)$.
	Then, we have
	\[
		h(x)=\frac{x}{2\tanh\frac{x}2}
		=\frac{x(e^{x}+1)}{2(e^{x}-1)}
	\]
	and
	\[
		h'(x)
		=\frac{(xe^x+e^{x}+1)(e^{x}-1)-x(e^{x}+1)e^{x}}{2(e^{x}-1)^2}
		=\frac{e^{2x}-2xe^x-1}{2(e^x-1)^2}.
	\]
	Because $(e^{2x}-2xe^x-1)'=2(e^{2x}-e^x-xe^x)=2e^x(e^x-1-x)$ is valid,
	we have $h'(x)>0$ for $x>0$.
	Evidently, we also have $\lim_{x\searrow0}h(x)=1$.
	Thus, $g$ satisfies the above properties.
	
	Because $g$ is positive and increasing, $\int_0^1g(tx)\dd x\le g(t)$ is valid.
	Therefore, we have
	\begin{align*}
		\int_0^1 K(tx) \dd x
		&=\int_0^1 g(tx) \dd x
		-\int_0^1 \log\left(\frac{\pi}{4d}tx\right) \dd x\\
		&\le g(t) - \log\left(\frac{\pi}{4d}t\right) + 1\\
		&=K(t)+1
	\end{align*}
	as desired.
\end{proof}

\begin{lem}\label{lem:1-2}
	For arbitrary $x>0$,
	the following hold true.
	\[
		K\left(\frac{x}2\right)\le K(x)+\log 2.
	\]
\end{lem}

\begin{proof}
	By the definition of $K$,
	it suffices to show that
	$\tanh x \le 2\tanh \frac{x}2$.
	Indeed, we have
	\begin{align*}
		\frac{2\tanh \frac{x}2}{\tanh x}
		=\frac{2(e^x-1)}{e^x+1}\cdot\frac{e^{2x}+1}{e^{2x}-1}
		=\frac{2(e^{2x}+1)}{(e^x+1)^2}
		\ge \frac{(e^{2x}+1)+2e^x}{(e^x+1)^2}=1,
	\end{align*}
	where we have used $e^{2x}+1\ge 2e^x$ (AM-GM inequality).
\end{proof}

We can now prove the first theorem.

\begin{proof}[Proof of Theorem \ref{main1}]
	The left inequality is from Theorem 3.4 and 3.5 in \citeasnoun{tanaka2018design}.
	
	Let us prove the right inequality.
	Let $a=(a_1, \ldots, a_n)$ (with $a_1<\cdots<a_n$) be the minimizer of the discrete energy,
	satisfying
	\[
		F_{K, Q}^\mathrm{D}(n)=\sum_{i\ne j}K(a_i-a_j)+\frac{n-1}n\sum_{i=1}^nQ(a_i).
	\]
	Let $\mu$ be a measure with a density function $p$ defined by
	\begin{align}
		p(x)=\begin{cases}
			\frac{n}{(n-1)(a_{i+1}-a_i)} & (x\in[a_i, a_{i+1}),\ i=1,\ldots,n-1),\\
			0 & (\text{otherwise}).
		\end{cases}
		\label{eq:feasible_solution_of_P}
	\end{align}
	Then, we have
	\begin{align}
		F_{K, Q}^\mathrm{C}(n)
		\le I_n^\mathrm{C}(\mu_n^*)
		\le I_n^\mathrm{C}(\mu).
		\label{eq:thm1-3-1}
	\end{align}
	In the following,
	we obtain an upper bound of $I_n^\mathrm{C}(\mu)$.
	First,
	we evaluate $\int_\R\int_\R K(x-y)\dd\mu(x)\dd\mu(y)$.
	For $1\le k\le n-1$ and $y\in[a_k, a_{k+1})$,
	we have
	\begin{align*}
		\int_\R K(x-y)\dd\mu(x)
		&=\int_\R K(x-y)p(x)\dd x\\
		&=\sum_{i=1}^{n-1} \frac{n}{(n-1)(a_{i+1}-a_i)}\int_{a_i}^{a_{i+1}}K(x-y)\dd x\\
		&=\frac{n}{n-1}\sum_{i=1}^{n-1}\int_0^1K(a_i+(a_{i+1}-a_i)z-y)\dd z.
	\end{align*}
	Here, because $y\in[a_k, a_{k+1})$, for $i\not\in\{k-1, k, k+1\}$,
	the convexity and monotonicity of $K$ over $(-\infty, 0)$ or $(0, \infty)$ shows that
	\[
		\int_0^1K(a_i+(a_{i+1}-a_i)z-y)\dd z\le
		\begin{cases}
			\frac12\left(K(a_i-a_k)+K(a_{i+1}-a_k)\right) &(i\le k-2),\\
			\frac12\left(K(a_i-a_{k+1})+K(a_{i+1}-a_{k+1})\right) &(i\ge k+2).
		\end{cases}
	\]
	Therefore, by considering that $K$ is non-negative, we have
	\begin{align}
		\sum_{i\ne k-1,k,k+1}\int_0^1K(a_i+(a_{i+1}-a_i)z-y)\dd z
		&\le \sum_{j\le k-2}K(a_j-a_k)+\sum_{j\ge k+3}K(a_j-a_{k+1})\nonumber\\
		&\quad +\frac12\left( K(a_{k-1}-a_k)+ K(a_{k+2}-a_{k+1}) \right)
		\label{eq:thm1-3-2}
	\end{align}
	Here, the terms that include an index of $a$ outside the domain $\{1,\ldots,n\}$
	are void.
	Next, we consider the cases $i=k\pm1$.
	If $k-1\ge1$ is valid, we have
	\begin{align}
		\int_0^1K(a_{k-1}+(a_k-a_{k-1})z-y)\dd z
		&\le \int_0^1K(a_{k-1}+(a_k-a_{k-1})z-a_k)\dd z\nonumber\\
		&=\int_0^1K((a_k-a_{k-1})w)\dd w\nonumber\\
		&\le K(a_k-a_{k-1}) + 1 =K(a_{k-1}-a_k)+1.
		\label{eq:thm1-3-3}
	\end{align}
	Similarly, if $k+2\le n$ is valid, we have, by Lemma \ref{lem:1-1},
	\begin{align}
		\int_0^1K(a_{k+1}+(a_{k+2}-a_{k+1})z-y)\dd z
		\le K(a_{k+2}-a_{k+1})+1.
		\label{eq:thm1-3-4}
	\end{align}
	Finally, we deal with the case $i=k$.
	We show that the integral
	\[
		L_k(y):=\int_0^1 K(a_k+(a_{k+1}-a_k)z-y)\dd z
	\]
	is maximized at $y=\frac{a_k+a_{k+1}}2$ (over $y\in[a_k, a_{k+1})$).
	If we define $t:=\frac{y-a_k}{a_{k+1}-a_k}$ ($t\in[0, 1)$),
	the following holds true.
	\[
		L_k(y)
		=\int_0^t K((a_{k+1}-a_k)w)\dd w + \int_0^{1-t} K((a_{k+1}-a_k)w)\dd w.
	\]
	For $t<\frac12$,
	we have
	\begin{align*}
		&L_k\left(\frac{a_k+a_{k+1}}2\right)-L_k(y)\\
		&=\int_t^{\frac12}K((a_{k+1}-a_k)w)\dd w-\int_{\frac12}^{1-t}K((a_{k+1}-a_k)w)\dd w\\
		&=\int_0^{\frac12-t}\left(
			K((a_{k+1}-a_k)(t+w))-K\left((a_{k+1}-a_k)\left(\frac12+w\right)\right)
		\right)\dd w > 0.
	\end{align*}
	By symmetry, $L_k(y)< L_k\left(\frac{a_k+a_{k+1}}2\right)$ is valid for $t>\frac12$.
	Therefore, by Lemma \ref{lem:1-1} and \ref{lem:1-2},
	\begin{align}
		\int_0^1 K(a_k+(a_{k+1}-a_k)z-y)\dd z
		&\le L_k\left(\frac12\right)\nonumber\\
		&=2\int_0^{\frac12}K((a_{k+1}-a_k)w)\dd w\nonumber\\
		&=\int_0^1K\left(\frac{a_{k+1}-a_k}2v\right) \dd v\nonumber\\
		&\le K\left(\frac{a_{k+1}-a_k}2\right)+1\nonumber\\
		&\le K(a_{k+1}-a_k)+1+\log 2
		\label{eq:thm1-3-5}
	\end{align}
	By (\ref{eq:thm1-3-2})--(\ref{eq:thm1-3-5}),
	we have the bound
	\begin{align}
		\left(\frac{n-1}n\right)^2&\int_{a_k}^{a_{k+1}}\int_\R K(x-y)\dd\mu(x)\dd\mu(y)
		\nonumber\\
		&\le\frac{n-1}n\sup_{y\in[a_k, a_{k+1})}\int_\R K(x-y)\dd\mu(x)\nonumber\\
		&\le \sum_{j\le k-2}K(a_j-a_k)
			+\sum_{j\ge k+3}K(a_j-a_{k+1}) + 3+\log2\nonumber\\
		&\quad+\frac32K(a_{k-1}-a_k)+\frac12K(a_k-a_{k+1})
		+\frac12K(a_{k+1}-a_k)+\frac32K(a_{k+2}-a_{k+1}).\nonumber
	\end{align}
	Considering the sum of the right-hand side with respect to $k=1,\ldots,n-1$,
	the coefficient of each $K(a_i-a_j)$ with $|i-j|\ge2$ is at most $1$,
	and that of $K(a_i-a_j)$ with $|i-j|=1$ is at most $2$ ($=\frac12+\frac32$),
	where we have distinguished $K(a_i-a_j)$ from $K(a_j-a_i)$.
	Therefore, we have
	\begin{align}
		\left(\frac{n-1}n\right)^2\int_\R\int_\R K(x-y)\dd\mu(x)\dd\mu(y)
		\le 2\sum_{i\ne j}K(a_i-a_j)+(n-1)(3+\log 2).
		\label{eq:thm1-3-6}
	\end{align}
	
	Let us now evaluate the second term of $I_n^\mathrm{C}(\mu)$,
	i.e., $\int_\R Q(x)\dd\mu(x)$.
	By the convexity of $Q$,
	we have
	\begin{align*}
		\int_\R Q(x)\dd\mu(x)
		&=\frac{n}{n-1}\sum_{i=1}^{n-1}\int_0^1Q(a_i+(a_{i+1}-a_i)z)\dd z\\
		&\le\frac{n}{n-1}\sum_{i=1}^{n-1}\max\{Q(a_i), Q(a_{i+1})\}.
	\end{align*}
	It should be noted here that there are no duplicates for $\max\{Q(a_i), Q(a_{i+1})\}$,
	i.e., it is impossible for $Q(a_{i+1})$ to be $\max\{Q(a_i), Q(a_{i+1}), Q(a_{i+2})\}$,
	by the strong convexity.
	Therefore,
	the following holds true.
	\begin{align}
		\int_\R Q(x)\dd\mu(x)
		\le \frac{n}{n-1}\sum_{i=1}^{n}Q(a_i).
		\label{eq:thm1-3-7}
	\end{align}
	Combining (\ref{eq:thm1-3-6}) and (\ref{eq:thm1-3-7}),
	we obtain
	\begin{align*}
		I_n^\mathrm{C}(\mu)
		&=\int_\R\int_\R K(x-y)\dd\mu(x)\dd\mu(y)
		+2\int_\R Q(x)\dd\mu(x)\\
		&\le 2\left(\frac{n}{n-1}\right)^2\sum_{i\ne j}K(a_i-a_j)
		+\frac{2n}{n-1}\sum_{i=1}^nQ(a_i)+
		\frac{n^2}{n-1}(3+\log2)\\
		&=2\left(\frac{n}{n-1}\right)^2\left(\sum_{i\ne j}K(a_i-a_j)
		+\frac{n-1}{n}\sum_{i=1}^nQ(a_i)\right)+
		\frac{n^2}{n-1}(3+\log2)\\
		&=2\left(\frac{n}{n-1}\right)^2 F_{K,Q}^\mathrm{D}(n)
		+\frac{n^2}{n-1}(3+\log2).
	\end{align*}
	Now, using (\ref{eq:thm1-3-1}), we reach the conclusion.
\end{proof}


\section{Duality theorem for convex programming of measures}\label{sec:duality_theorem}

The following definition is a variant of the
existing definitions of positive definite kernel
\cite{stewart1976positive,jaming2009extremal,sriperumbudur2010hilbert}.

\begin{dfn}\label{def2-1}
	Let $X$ be a topological space.
	A non-negative measurable function $k:X\times X\to\R_{\ge0}\cup\{\infty\}$
	is called {\it positive semi-definite in measure} if it satisfies
	\begin{align}
		\int_X\int_Xk(x,y)\dd\mu(x)\dd\mu(y)\ 
		+ & \ \int_X\int_Xk(x,y)\dd\nu(x)\dd\nu(y)\nonumber\\
		&\ge
		\int_X\int_Xk(x,y)\dd\mu(x)\dd\nu(y)
		+\int_X\int_Xk(x,y)\dd\nu(x)\dd\mu(y)
		\label{eq:def2-1}
	\end{align}
	for arbitrary (positive) $\sigma$-finite Borel measures $\mu, \nu$ on $X$.
\end{dfn}

\begin{rem}\label{rem2-2}
	Let $k$ be positive semi-definite in measure.
	Considering the Hahn-Jordan decomposition of a signed measure,
	we have
	\[
		\int_X\int_Xk(x,y)\dd|\mu|(x)\dd|\mu|(y)<\infty
		\Longrightarrow
		\int_X\int_Xk(x,y)\dd\mu(x)\dd\mu(y)\ge0
	\]
	for an arbitrary signed Borel measure $\mu$ on $X$ with $|\mu|$ being $\sigma$-finite,
	where $|\mu|$ denotes the total variation of $\mu$.
	This is the generalization of the ordinary positive semi-definiteness.
	Notice that this non-negativity holds for a wider class of ``measure".
	Indeed, if we define
	\[
		\S_k:=\left\{(\mu_+, \mu_-) \,\middle|\,
		\begin{array}{c}
			\text{$\mu_+$ and $\mu_-$ are $\sigma$-finte Borel measures}\\
			\int_X\int_Xk(x,y)\dd\mu_+(x)\dd\mu_+(y),
			\int_X\int_Xk(x,y)\dd\mu_-(x)\dd\mu_-(y)<\infty
		\end{array}\right\}
	\]
	and for each $\nu=(\nu_+,\nu_-)\in\S_k$ define
	\begin{align*}
		\int_X\int_Xk(x,y)\dd\nu(x)\dd\nu(y)
		&:=
		\int_X\int_Xk(x,y)\dd\nu_+(x)\dd\nu_+(y)\ 
		+  \ \int_X\int_Xk(x,y)\dd\nu_-(x)\dd\nu_-(y)\nonumber\\
		&\quad -
		\int_X\int_Xk(x,y)\dd\nu_+(x)\dd\nu_-(y)
		-\int_X\int_Xk(x,y)\dd\nu_-(x)\dd\nu_+(y),
	\end{align*}
	then this integral is well-defined and the generalization of quadratic forms for ordinary signed measures.
	We formally write $\nu=\nu_+-\nu_-$ in such a situation,
	and call it also the Hahn-Jordan decomposition of $\nu$.
\end{rem}

\begin{lem}\label{lem2-3}
	Let $K:\R\to\R_{\ge0}\cup\{+\infty\}$ be an even function.
	If $K\in L^1(\R)$ and $K$ is convex on $[0, \infty)$,
	and satisfies $\lim_{x\searrow0}K(x)=K(0)$,
	then $K(x-y)$ is positive semi-definite in measure.
\end{lem}

\begin{proof}
	Because $K$ is integrable and convex, $K$ is continuous over $(0, \infty)$
	and $\lim_{x\to\infty}K(x)=0$ holds true.
	If $K(0)<\infty$, $K$ becomes continuous and this type of function is called
	{\it P\'{o}lya-type}. P\'{o}lya-type functions are known to be a characteristic function of
	a positive bounded Borel measure,
	i.e., there exists a positive bounded measure $\alpha$ on $\R$ such that
	\begin{align}
		K(x)=\int_\R e^{-i\omega x} \dd\alpha(\omega)
		\label{eq:lem2-3-1}
	\end{align}
	is valid \cite{jaming2009extremal,polya1949remarks}.
	Let $\mu$ be a signed Borel measure with
	$\int_\R\int_\R K(x-y)\dd\mu(x)\dd\mu(y)$
	being finite and $|\mu|$ being $\sigma$-finite.
	Then, we can take a sequence of increasing Borel sets
	$A_1\subset A_2\subset \cdots\to\R$
	satisfying $|\mu|(A_k)<\infty$ for all $k$.
	Let $\mu=\mu_+-\mu_-$ be the Hahn-Jordan decomposition and
	$\mu_+^k:=\mu_+(A_k\cap \cdot)$, $\mu_-^k:=\mu_-(A_k\cap \cdot)$.
	For each $k$,
	by Fubini's theorem and (\ref{eq:lem2-3-1}), we have
	\[
		\int_\R\int_\R K(x-y)
		\dd(\mu_+^k-\mu_-^k)(x)\dd(\mu_+^k-\mu_-^k)(y)
		=\int_\R\left|\int_\R e^{-ikx}\dd(\mu_+^k-\mu_-^k)\right|^2\dd\alpha(\omega)\ge0.
	\]
	This can be rewritten as
	\begin{align*}
		\int_\R\int_\R K(x-y)
		\dd\mu_+^k(x)\dd\mu_+^k(y)
		\ + & \int_\R\int_\R K(x-y)
		\dd\mu_-^k(x)\dd\mu_-^k(y)\\
		&\ge2\int_\R\int_\R K(x-y)
		\dd\mu_+^k(x)\dd\mu_-^k(y)
	\end{align*}
	and the monotone convergence theorem leads to the desired inequality,
	as, for example,
	\[
		\int_\R\int_\R K(x-y)
		\dd\mu_+^k(x)\dd\mu_-^k(y)
		=\int_\R\int_\R 1_{A_k\times A_k}(x, y)K(x-y)\dd\mu_+(x)\dd\mu_-(y)
	\]
	is valid.

	Let us consider the case $K(0)=\infty$.
	In this case, $K$ is continuous on $(0, \infty)$
	and has a limit $\lim_{x\searrow 0}K(x)$
	For any $\ve>0$,
	define
	\[
		K_\ve(x):=\frac1\ve\int_0^\ve K(|x|+z)\dd z,\quad x\in\R.
	\]
	Then, by $K\in L^1(\R)$, $K$ is bounded everywhere by $\ve^{-1}\| K\|_{L^1}$.
	Moreover, $K_\ve$ is still convex,
	such that $K_\ve(x-y)$ is positive semi-definite in measure.
	Now, the continuity of $K$ leads to
	\[
		K_\ve(x)
		=\int_0^1K(|x|+\ve z) \dd z
		\nearrow K(|x|)=K(x)\quad (\ve\searrow0)
	\]
	by the monotone convergence theorem.
	Applying the monotone convergence theorem to both sides of (\ref{eq:def2-1}) with $K=K_\ve$,
	we obtain the conclusion.
\end{proof}

The function $K=-\log\left|\tanh\left(\frac\pi{4d}\cdot\right)\right|$ satisfies the condition of 
Lemma \ref{lem2-3}.
Thus, we can observe the optimization problem
\begin{align}
	(\mathrm{P})\quad
	\begin{array}{lr}\text{minimize}&\displaystyle
	\int_\R\int_\R K(x-y)\dd\mu(x)\dd\mu(y) + 2\int_\R Q(x)\dd\mu(x)
	\nonumber\\
	\text{subject to}&\mu\in\mathcal{M}(\R, n)
	\end{array}
\end{align}
as convex quadratic programming.
We can analogously make the dual problem to the finite-dimensional case in
\citeasnoun{dorn1960duality}, as
\begin{align}
	(\mathrm{D})\quad
	\begin{array}{lr}\text{maximize}&\displaystyle -\int_\R\int_\R K(x-y)\dd\nu(x)\dd\nu(y)
	+ 2ns\nonumber\\
	\text{subject to}&\displaystyle \nu\in \S_K,\ s-\int_\R K(\cdot-y)\dd\nu(y)\le Q.
	\end{array}
\end{align}
Note that this is a rigorous version of problem (D) in \eqref{eq:pre_problem_D}. 
It should be noted here that we have not justified (D) as a formal (topologically) dual problem.
There are arguments limited to the optimization of Radon measure over compact space
\cite{ohtsuka1966a,ohtsuka1966b,wu2001}.
While they are on quadratic programming problems,
there exist more general theories on duality,
such as \citeasnoun{isii1964},
von Neumann's minimax theorem \cite{neumann1928,sion1958general}
and Fenchel-Rockafellar duality theorem \cite{rockafellar1966extension,villani2003topics}.
However, as it is essential that our duality can treat infinite measure $\nu$
with unbounded support
(we indeed later use such a measure as a dual feasible solution),
it is difficult to just apply existing studies and check all the conditions
for (D) to be a topologically dual problem.
Therefore, we here do not go deeper in this aspect,
but just prove the assertion of Theorem \ref{thm2-4}.
This assertion is sufficient to derive a lower bound
of the optimal value of (P),
which is our objective.

In the following, we demonstrate that the weak duality and strong duality are
still valid in this infinite-dimensional primal-dual pair.
It should be noted that $s=0$, $\nu\equiv0$ is a trivial feasible solution of (D)
such that there exists an optimal value of (D).

\begin{thm}\label{thm2-4}
	The optimal value of (D) is equal to the optimal value of (P).
\end{thm}

\begin{proof}
	First, we present the weak duality.
	Let $\mu$ and $(\nu, s)$ be feasible solutions of (P) and (D), respectively,
	and $\nu=\nu_+-\nu_-$ be the Hahn-Jordan decomposition.
	If we write $\langle\alpha, \beta\rangle_K
	:=\int_\R\int_\R K(x-y)\dd\alpha(x)\dd\beta(y)$
	for measures $\alpha$ and $\beta$,
	\[
		\kk\nu\nu=\kk{\nu_+}{\nu_+}+\kk{\nu_-}{\nu_-}-2\kk{\nu_+}{\nu_-}
	\]
	holds true.
	Because $\kk\mu\mu, \kk{\nu_+}{\nu_+}, \kk{\nu_-}{\nu_-}<\infty$,
	we have 
	$\kk\mu{\nu_+}, \kk\mu{\nu_-}, \kk{\nu_+}{\nu_-}<\infty$
	by $K$'s positive semi-definiteness in measure.
	Therefore, we have
	\begin{align*}
		&\left(
			\int_\R\int_\R K(x-y)\dd\mu(x)\dd\mu(y) + 2\int_\R Q(x)\dd\mu(x)
		\right)
		-\left(
			-\int_\R\int_\R K(x-y)\dd\nu(x)\dd\nu(y) + 2ns
		\right)\\
		&=\kk\mu\mu+\kk\nu\nu+2\int_\R (Q(x)-s)\dd\mu(x)\\
		&\ge\kk\mu\mu+\left(
			\kk{\nu_+}{\nu_+}+\kk{\nu_-}{\nu_-}-2\kk{\nu_+}{\nu_-}
		\right)
		+2\int_\R \left(-\int_\R K(x-y)\dd\nu(y)\right)\dd\mu(x)\\
		&=\kk\mu\mu+\kk{\nu_+}{\nu_+}+\kk{\nu_-}{\nu_-}-2\kk{\nu_+}{\nu_-}
		-2\kk\mu{\nu_+}+2\kk\mu{\nu_-}\\
		&=\kk{\mu+\nu_-}{\mu+\nu_-}+\kk{\nu_+}{\nu_+}-2\kk{\mu+\nu_-}{\nu_+}\ge0
	\end{align*}
	by the positive semi-definiteness in measure.
	This indicates the weak duality.
	
	To prove the strong duality,
	we construct the optimal solution of (D) using that of (P).
	By Theorem 2.4 in \citeasnoun{tanaka2018design},
	$\mu^*$, the optimal solution of (P), satisfies
	\begin{align}
		\int_\R K(x-y)\dd\mu^*(y)+Q(x)\ge \frac{F_{K, Q}^\mathrm{C}(n)}n
		\label{eq:thm2-4-1}
	\end{align}
	for all $x\in\R$.
	Now, $\mu^*$ and $n^{-1}F^\mathrm{C}_{K, Q}(n)$ is a feasible solution for (D).
	Moreover, the equality of (\ref{eq:thm2-4-1}) is valid on the support of $\mu^*$,
	such that we have
	\begin{align*}
		&-\int_\R\int_\R K(x-y)\dd\mu^*(x)\dd\mu^*(y)+2n\frac{F^\mathrm{C}_{K, Q}(n)}n\\
		&=-\int_\R\int_\R K(x-y)\dd\mu^*(x)\dd\mu^*(y)
		+2\int_\R\left(Q(x)+\int_\R K(x-y)\dd\mu^*(y)\right)\dd\mu^*(x)\\
		&=\int_\R\int_\R K(x-y)\dd\mu^*(x)\dd\mu^*(y)+2\int_\R Q(x)\dd\mu^*(x).
	\end{align*}
	This shows the strong duality.
\end{proof}



\section{Proof of Theorem \ref{main2}}\label{sec5}

We can now give a lower bound of $F_{K, Q}^\mathrm{C}(n)$ by using the dual problem (D)
and prove Theorem \ref{main2}.
Let $\alpha>0$ be a constant and $f$ be the inverse Fourier transform of
\[
	\left(\F[f](\omega)=\right)\quad\frac{\omega}{\pi\tanh(d\omega)}
	\int_{-\alpha}^\alpha \left( Q(\alpha) - Q(x) \right) e^{-i\omega x}\dd x
\]
Along with this, $f$ is $L^2$-integrable by Theorem 4.4 in \citeasnoun{tanaka2017potential}.
Here, the Fourier transform of a function $g\in L^1(\R)\cap L^2(\R)$ is defined by
\[
	\F[g](\omega):=\int_\R g(x)e^{-i\omega x}\dd x
\]
and for the whole space $L^2(\R)$,
$\F[\cdot]$ is defined as the continuous extension of $\F[\cdot]|_{L^1\cap L^2}$.
Because $Q(x)$ is even by the assumption,
$f$ is an inverse Fourier transform of an even real function,
so that $f$ itself is an even real function.
Then, the formula \citeaffixed{fourier_table}{p.43, 7.112 in}
\[
	\F\left[
		\log\left|\tanh\left(\frac\pi{4d}\cdot\right)\right|
	\right](\omega)
	=-\frac\pi\omega\tanh(d\omega)
\]
leads to the (almost everywhere) equation
\begin{align}
	\F\left[\int_\R K(x-y)f(y)\dd y\right](\omega)
	=\F[K](\omega)\cdot\F[f](\omega)
	=\int_{-\alpha}^\alpha \left( Q(\alpha) - Q(x) \right) e^{-i\omega x}\dd x,
	\label{eq:2-1}
\end{align}
where $K\in L^1(\R)\cap L^2(\R)$ (see Appendix A)
and $f\in L^2(\R)$ are used for the justification of the first equality.
The integrability of $K(x-\cdot)f(\cdot)$ comes from $K, f\in L^2(\R)$ and
by Minkowski's integral inequality \citeaffixed[Theorem 202]{inequalities}{see, e.g.,}, we have
\begin{align*}
	\left\|\int_\R K(\cdot-y)f(y)\dd y\right\|_{L^2}
	&=\left\|\int_\R K(y)f(\cdot-y)\dd y\right\|_{L^2}\\
	&\le\int_\R K(y)\| f(\cdot-y)\|_{L^2}\dd y\\
	&=\|K\|_{L^1}\|f\|_{L^2}<\infty.\\
\end{align*}
Considering the inverse Fourier transform of (\ref{eq:2-1}),
we also have
\[
	\int_\R K(x-y)f(y)\dd y=1_{[-\alpha, \alpha]}(x)(Q(\alpha)-Q(x)).
\]
It should be noted that $f(x)\dd x \in \S_K$ follows from the inequality
	\[
		\int_\R\int_\R K(x-y)|f(x)f(y)|\dd x\dd y\le
		\|K*f\|_{L^2}\|f\|_{L^2}\le \|K\|_{L^1}\|f\|_{L^2}^2<\infty.
	\]
These two relations imply that
$(f(x)\dd x, Q(\alpha))$ is a feasible solution of (D).
We can now evaluate the value of the objective function of (D).
Let us define
\begin{align}
	F(\alpha):=-\int_\R\int_\R K(x-y)f(x)f(y)\dd x\dd y+2nQ(\alpha).
	\label{eq:2-2}
\end{align}
Because the first term can be considered as the inner product of
$K*f$ and $f$ in $L^2(\R)$,
it can be computed through the Fourier transform as
\begin{align}
	&\int_\R\int_\R K(x-y)f(x)f(y)\dd x\dd y\nonumber\\
	&=\frac1{2\pi}\int_\R\left(
		\frac{\omega}{\pi\tanh(d\omega)}
	\int_{-\alpha}^\alpha \left( Q(\alpha) - Q(x) \right) e^{-i\omega x}\dd x\ 
	\overline{\int_{-\alpha}^\alpha \left( Q(\alpha) - Q(x) \right) e^{-i\omega x}\dd x}\right)
	\dd\omega\nonumber\\
	&=\frac1{2\pi^2}\int_\R
		\frac{\omega}{\tanh(d\omega)}
	\left|\int_{-\alpha}^\alpha \left( Q(\alpha) - Q(x) \right) e^{-i\omega x}\dd x\right|^2
	\dd\omega.
	\label{eq:2-3}
\end{align}
Let $G(\alpha)$ be the value of the right-hand side.
$G(\alpha)$ can be decomposed into two parts, which are defined as
\[
	G_1(\alpha):=\frac1{2\pi^2}\int_{-1}^1\frac{\omega}{\tanh(d\omega)}
	\left|\int_{-\alpha}^\alpha \left( Q(\alpha) - Q(x) \right) e^{-i\omega x}\dd x\right|^2
	\dd\omega
\]
and
\[
	G_2(\alpha):=\frac1{2\pi^2}\int_{[-1, 1]^c}\frac{\omega}{\tanh(d\omega)}
	\left|\int_{-\alpha}^\alpha \left( Q(\alpha) - Q(x) \right) e^{-i\omega x}\dd x\right|^2
	\dd\omega.
\]
We first evaluate $G_1$.
Because the function $\omega/\tanh(d\omega)$ is monotonically increasing in $[0, \infty)$
(see the proof of Lemma \ref{lem:1-1}),
we have
\begin{align}
	G_1(\alpha)
	&\le\frac1{\pi\tanh(d)}\cdot \frac1{2\pi}\int_\R
	\left|\int_{-\alpha}^\alpha \left( Q(\alpha) - Q(x) \right)
	e^{-i\omega x}\dd x\right|^2\dd\omega \nonumber\\
	&=\frac1{\pi\tanh(d)}\|1_{[-\alpha, \alpha]}(x)(Q(\alpha)-Q(x))\|_{L^2}^2\nonumber\\
	&\le \frac2{\pi\tanh(d)}\ \alpha Q(\alpha)^2.
\end{align}
Next, we similarly evaluate $G_2$.
By integration by parts, we get
\[
	\omega\int_{-\alpha}^\alpha(Q(\alpha)-Q(x))e^{-i\omega x}\dd x
	=-\frac1i\int_{-\alpha}^\alpha Q'(x)e^{-i\omega x}\dd x.
\]
Thus, we have
\begin{align}
	G_2(\alpha)&=\frac1{2\pi^2}\int_{[-1, 1]^c}\frac1{\omega\tanh(d\omega)}
	\left|
		\int_{-\alpha}^\alpha Q'(x)e^{-i\omega x}\dd x
	\right|^2\dd\omega\nonumber\\
	&\le \frac1{\pi\tanh(d)}\|1_{[-\alpha,\alpha]}(x)Q'(x)\|_{L^2}^2\nonumber\\
	&\le \frac2{\pi\tanh(d)}\ \alpha Q'(\alpha)^2.
\end{align}
Finally, we reach the evaluation
\[
	G(\alpha)\le \frac{2\alpha}{\pi\tanh(d)}\left(Q(\alpha)^2+Q'(\alpha)^2\right),
	\quad
	F(\alpha)\ge 2nQ(\alpha)-\frac{2\alpha}{\pi\tanh(d)}\left(Q(\alpha)^2+Q'(\alpha)^2\right).
\]
By letting $\alpha_n$ satisfy
\[
	\frac{2\alpha_n}{\pi\tanh(d)}\frac{Q(\alpha_n)^2+Q'(\alpha_n)^2}{Q(\alpha_n)}\le n,
\]
we get $nQ(\alpha_n)$ as a lower bound for the optimal value of (P).
For such $\alpha_n$,
we finally have
\[
	nQ(\alpha_n)
	\le I^\mathrm{C}_{K, Q}(\mu^*)
	\le 2F_{K, Q}^\mathrm{C}(n)
\]
and this is equivalent to the assertion of Theorem \ref{main2}.

\section{Examples of convergence rates for several $Q(x)$'s}\label{sec:numerical_examples}

Although the asymptotic rates given in \citeasnoun[Section 4.3]{tanaka2017potential}
are derived through mathematically informal arguments,
we here demonstrate that those rates roughly coincide with
the bound in Theorem \ref{main2}.

\begin{eg}\label{eg:SE}
(The case $w$ is a single exponential)
Consider the case
\[
	w(x)=\exp\left(-(\beta|x|)^\rho\right),\quad
	Q(x)=(\beta|x|)^\rho,
\]
for $\beta>0$ and $\rho\ge1$.
In this case, for a sufficiently large $\alpha$ (satisfying $\alpha\ge\rho$),
we have
\[
	\frac{2\alpha}{\pi\tanh(d)}\frac{Q(\alpha)^2+Q'(\alpha)^2}{Q(\alpha)}
	=\frac{2\alpha}{\pi\tanh(d)}\frac{(\beta\alpha)^{2\rho}
	+(\beta\rho)^2(\beta\alpha)^{2(\rho-1)}}{(\beta\alpha)^\rho}
	\le \frac{4\beta^\rho\alpha^{\rho+1}}{\pi\tanh(d)}
\]
and $\alpha_n$ can be taken as
\begin{align}
	& \alpha_n=\left(\frac{\pi\tanh(d)}{4\beta^\rho}n\right)^{\frac1{\rho+1}},
	\notag \\
	& \frac{Q(\alpha_n)}2
	=\frac12\beta^\rho\left(\frac{\pi\tanh(d)}{4\beta^\rho}n\right)^{\frac\rho{\rho+1}}
	\qquad \left( =\Theta\left( \beta^{\frac\rho{\rho+1}}n^{\frac\rho{\rho+1}}\right) \right),
	\label{eq:SE_th_rate}
\end{align}
for sufficiently large $n$.
This rate roughly coincides with (4.37) in \citeasnoun{tanaka2017potential}.
\end{eg}

\begin{eg}\label{eg:DE}
(The case $w$ is a double exponential)
Consider the case
\[
	w(x)=\exp\left(-\beta\exp(\gamma|x|)\right),
	\quad
	Q(x)=\beta\exp(\gamma|x|),
\]
for $\beta,\gamma>0$.
In this case,
\[
	\frac{2\alpha}{\pi\tanh(d)}\frac{Q(\alpha)^2+Q'(\alpha)^2}{Q(\alpha)}
	=\frac{2\alpha\beta(1+\gamma^2)\exp(\gamma\alpha)}{\pi\tanh(d)}
\]
is valid.
Let $\alpha_n>0$ satisfy that the right-hand side is equal to $n$.
Then, we have
\[
	\gamma\alpha_n=W\left(
		\frac{\pi\tanh(d)\gamma}{2\beta(1+\gamma^2)}n
	\right)
	\qquad \left( \sim\log\left(\frac\gamma{\beta(1+\gamma^2)}n\right) \right),
\]
where $W$ is Lambert's W function, i.e., the inverse of $x\mapsto xe^x$.
Using this, we get
\begin{align}
	\frac{Q(\alpha_n)}2
	&=\frac{\beta}{2\gamma\alpha_n}\cdot\gamma\alpha_n\exp(\gamma\alpha_n)
	=\frac{\beta}{2\gamma\alpha_n}\frac{\pi\tanh(d)\gamma}{2\beta(1+\gamma^2)}n
	=\frac{\pi\tanh(d)}{4(1+\gamma^2)}\frac{n}{\alpha_n} \notag \\
	& = \frac{\pi\tanh(d) \gamma}{4(1+\gamma^2)}\frac{n}{W\left(
		\frac{\pi\tanh(d)\gamma}{2\beta(1+\gamma^2)}n
	\right)}  
	\qquad
	\left(
	\sim\frac{\pi\tanh(d)\gamma}{4(1+\gamma^2)}
	\frac{n}{\log\left( \frac\gamma{\beta(1+\gamma^2)}n\right)}
	\right).
	\label{eq:DE_th_rate}
\end{align}
This rate roughly coincides with the asymptotic order (4.44)
in \citeasnoun{tanaka2017potential}
for each fixed constant $\gamma$.
\end{eg}

\begin{remark}
We choose the weight functions in Examples \ref{eg:SE} and \ref{eg:DE} for simplicity
although they are not (necessarily) analytic in the strip region $\mathcal{D}_{d}$ for any $d > 0$.
This is because we just need their asymptotic properties for finding $\alpha_{n}$. 
\end{remark}



\section{Conclusion}\label{sec:conclusion}
In this study,
we analyzed the approximation method proposed by \citeasnoun{tanaka2018design}
over weighted Hardy spaces $\H^\infty(\D_d, w)$.
We provided (1) proof of the fact that the approximation formulas are nearly optimal
from the viewpoint of minimum worst-case error $E_n^{\min}(\H^\infty(\D_d, w))$;
and (2) upper bounds of $E_n^{\min}(\H^\infty(\D_d))$ to evaluate the convergence rates
of approximation errors with $n\to\infty$.
To obtain (2),
we introduced the concept ``positive semi-definite in measure''
and by using this, provided a lower bound for $F_{K, Q}^\mathrm{C}(n)$.
We also compared the given bounds with those mentioned in the study by \citeasnoun{tanaka2017potential},
and demonstrated that they have the same convergence rate with $n\to\infty$.

The new bounds do not indicate that the approximation formulas in \citeasnoun{tanaka2018design} are optimal. 
Another method to bound the error is recently considered by \citeasnoun{van2022convergence}, 
although their bound do not show the optimality, either. 
We need tighter bounds to show the optimality, which may require more sophisticated analysis. 
We leave such analysis to future work. 


\section*{Acknowledgements}
The authors are grateful to Ryunosuke Oshiro for his comment on signed measures.
This study was supported by the
Japan Society for the Promotion of Science with KAKENHI (17K14241 to K.T.).

\bibliographystyle{dcu}
\bibliography{cite1}

\renewcommand{\appendixname}{Appendix }

\appendix

\def\thesection{Appendix \Alph{section}}

\section{Proof of $K \in L^1(\R)\cap L^2(\R)$}

It suffices to consider the case $d=\pi/4$,
i.e., $K(x)=-\log|\tanh(x)|$,
and prove
\[
	\int_0^\infty(-\log\tanh(x))\dd x < \infty \quad \text{and} \quad
	\int_0^\infty(-\log\tanh(x))^2\dd x < \infty
\]
as $K$ is even.
By variable transformation $y=\tanh(x)$
$\left(\Leftrightarrow\displaystyle x=\frac12\log\frac{1+y}{1-y}\right)$,
we have
\[
	\int_0^\infty(-\log\tanh(x))\dd x
	=\int_0^1(-\log y)\frac1{1-y^2}\dd y.
\]
Additionally, by setting $z=-\log y$,
we get
\[
	\int_0^1(-\log y)\frac1{1-y^2}\dd y
	=\int_0^\infty \frac{ze^{-z}}{1-e^{-2z}}\dd z.
\]
Performing the same variable transformations,
we have
\[
	\int_0^\infty(-\log\tanh(x))^2\dd x
	=\int_0^\infty\frac{z^2e^{-z}}{1-e^{-2z}}\dd z.
\]
Because $z\ge z^2$ over $(0, 1]$ and $z\le z^2$ over $[1, \infty)$,
it suffices to show that
\[
	\int_0^1 \frac{ze^{-z}}{1-e^{-2z}}\dd z<\infty \quad \text{and} \quad
	\int_1^\infty\frac{z^2e^{-z}}{1-e^{-2z}}\dd z<\infty.
\]
For the former,
because $e^{2z}-1\ge 2z$ is valid,
we have
\[
	\int_0^1\frac{ze^{-z}}{1-e^{-2z}}\dd z
	=\int_0^1\frac{ze^z}{e^{2z}-1}\dd z
	\le\int_0^1\frac{e^z}2\dd z<\infty.
\]
For the latter,
we have
\[
	\int_1^\infty\frac{z^2e^{-z}}{1-e^{-2z}}\dd z
	\le\frac1{1-e^{-2}}\int_1^\infty z^2e^{-z}\dd z
	\le\frac1{1-e^{-2}}\Gamma(3)<\infty.
\]
Therefore,
we finally get the result $K\in L^1(\R)\cap L^2(\R)$.


\makeatletter

\def\pct{\expandafter\@gobble\string\%}

\immediate\write\@auxout{\pct\space This is a test line.\pct }

\end{document}